\documentclass[12pt]{article}
\usepackage{amsfonts, amssymb, amsthm}

\usepackage[utf8]{inputenc}

\newcommand{\UU}{\mathcal{U}}
\newcommand{\B}{\mathcal{B}}

\newtheorem{theorem}{Theorem}
\newtheorem{proposition}{Proposition}
\newtheorem{lemma}{Lemma}

\theoremstyle{definition}
\newtheorem{example}{Example}

\title{Thin subsets of groups}
\author{I.V.Protasov, S. Slobodianiuk}
\date{}

\begin{document}

\maketitle

\begin{abstract}
For a group $G$ and a natural number $m$, a subset $A$ of $G$ is called $m$-thin if, for each finite subset $F$ of $G$, there exists a finite subset $K$ of $G$ such that $|Fg\cap A|\leqslant m$ for every $g\in G\setminus K$. We show that each $m$-thin subset of a group $G$ of cardinality $\aleph_n$, $n= 0,1,\ldots$ can be partitioned into $\leqslant m^{n+1}$ 1-thin subsets. On the other side, we construct a group $G$ of cardinality $\aleph_\omega$ and point out a 2-thin subset of $G$ which cannot be finitely partitioned into 1-thin subsets.
\end{abstract}

 Let $G$ be a group, $\kappa$ and $\mu$ be cardinals, $|G|\geqslant \kappa \geqslant \aleph_0$ and $\mu \leqslant \kappa$, $[G]^{<\kappa} = \{X\subset G: |X| < \kappa\}$.

We say that a subset $A$ of $G$ is {\it $(\kappa,\mu)$-thin } if, for every $F\in [G]^{<\kappa}$, there exists $K\in [G]^{<\kappa}$ such that 

$$ |Fg\cap A|\leqslant \mu$$
for each $g\in G\setminus K$.

If $\kappa$ is regular then $A$ (see Lemma 1) is $(\kappa,1)$-thin if and only if, for each $g\in G$, $g\ne e$, $e$ is the identity of $G$, we have 
$$|\{a\in A: ga\in A\}|< \kappa.$$

An $(\aleph_0,1)$-thin subset is called {\it thin }. For thin subsets, its modifications and applications see \cite{b1}, \cite{b2}, \cite{b3}, \cite{b4}, \cite{b5}, \cite{b6}, \cite{b7}. For $m\in \mathbb{N}$, the $(\aleph_0, m)$-thin subsets appeared in \cite{b3} under name {\it $m$-thin} in attempt to characterize the ideal in the Boolean algebra of subsets of $G$ generated by the family of thin subsets of $G$. If a subset $A$ of $G$ is a union of $m$ thin subsets then $A$ is $m$-thin. On the other hand, if $G$ is countable and $A$ is $m$-thin then $A$ can be partitioned into $\leqslant m$ thin subsets. Thus, the ideal generated by thin subsets of a countable group $G$ coincides with the family of all $m$-thin, $m\in \mathbb{N}$ subsets of $G$. Does this characterization remain true for all infinite groups? In other words, can every $m$-thin subset of an uncountable group $G$ be partitioned in $m$ (finitely many) thin subsets? In this paper we give answer to these questions.

The paper consists of 5 sections. In the first section we see that the thin subsets can be defined in the much more general context of balleans, the counterparts of the uniform topological spaces. From this point of view, a thin subset is a counterpart of a uniformly discrete subset of a uniform space. As a corollary of some ballean statement (Theorem 1), we get that, for each infinite regular cardinal $\kappa$ and each $m\in \mathbb{N}$, every $(\kappa, m)$-thin subset of a group $G$ of cardinality $\kappa$ can be partitioned into $\leqslant m$ $(\kappa,1)$-thin subsets.

In section 2 we extend this statement showing (Theorem 3) that, for every infinite regular cardinal $\kappa$, $m\in \mathbb{N}$ and $n\in \omega$, each $(\kappa, m)$-thin subset of a group $G$ of cardinality $\kappa^{+n}$ can be partitioned into $\leqslant m^{n+1}$ $(\kappa,1)$-thin subsets. Here, $\kappa^{+0} = \kappa, \kappa^{+(n+1)}= (\kappa^n)^+$. In particular, every $m$-thin subset of a group $G$ of cardinality $\aleph_n$ can be partitioned into $\leqslant m^{n+1}$ thin subsets. Clearly, in this case the ideal generated by thin subsets also coincides with the family of all $m$-thin subsets, $m\in \mathbb{N}$. In Theorem 4 we describe $(\kappa, \mu)$-thin groups that can be partitioned into $\mu$ $(\kappa,1)$-subsets.

In section 3 one can find two auxiliary combinatorial theorems (of independent interest!) on coloring of the square $G\times G$ of a group $G$ which will be used in the next section.

Answering a question from \cite{b3}, G.~Bergman constructed a group $G$ of cardinality $\aleph_2$ and a 2-thin subset $A$ of $G$ which cannot be partitioned into two thin subsets. With kind permission of the author, we reprint in section 4 his letter with this remarkable construction (Example 1). Then we modify the Bergman's construction to show (Example 2) that for each natural number $m\geqslant 2$ there exist a group $G_n$ of cardinality $\aleph_n$, $ n = \frac{m(m+1)}{2} -1$, and a 2-thin subset $A$ of $G$ which cannot be partitioned into $m$-thin subsets. And finally (Example 3), we point out a group $G$ of cardinality $\aleph_\omega$ and a 2-thin subset of $G$ which cannot be finitely partitioned into thin subsets.

We conclude the paper with some observations on interplay between thin subsets and ultrafilters in section 5.

\section{Ballean context}

A {\it ball structure} is a triple $\mathcal{B}=(X,P,B)$, where $X$, $P$ are non-empty sets and, for
any $x\in X$ and $\alpha\in P$, $B(x,\alpha)$ is a subset of $X$
which is called a \emph{ball of radius} $\alpha$ around $x$. It is
supposed that $x\in B(x,\alpha)$ for all $x\in X$ and $\alpha\in P$.
The set $X$ is called the {\it support} of $\mathcal{B}$, $P$ is
called the {\it set of radii}. Given any $x\in X, A\subseteq X,
\alpha\in P$ we put
$$
B^*(x,\alpha)=\{y\in X:x\in B(y,\alpha)\},\
B(A,\alpha)=\bigcup_{a\in A}B(a,\alpha).
$$

Following \cite{b8}, we say that a ball structure $\mathcal{B}=(X,P,B)$ is a {\it ballean} if 
\begin{itemize}
\item for any $\alpha,\beta\in P$,
there exist $\alpha',\beta'$ such that, for every $x\in X$,
$$B(x,\alpha)\subseteq B^*(x,\alpha'),\ B^*(x,\beta)\subseteq B(x,\beta');$$
\item for any $\alpha,\beta\in P$,
there exists $\gamma\in P$ such that, for every $x\in X$,
$$B(B(x,\alpha),\beta)\subseteq B(x,\gamma).$$
\end{itemize}

We note that a ballean can also be defined in terms of entourages of diagonal in $X\times X$. In this case it is called a coarse structure \cite{b9}.

A ballean $\B$ is called {\it connected} if, for any $x,y\in X$, there exists $\alpha\in P$ such that $y\in B(x,\alpha)$. All balleans under consideration are supposed to be connected. Replacing each ball $B(x,\alpha)$ to $B(x,\alpha) \cap B^*(x,\alpha)$, we may suppose that $B(x,\alpha) = B^*(x,\alpha)$ for all $x\in X,\alpha\in P$. A subset $Y\subseteq X$ is called {\it bounded} if there exist $x\in X$ and $\alpha\in P$ such that $Y\subseteq B(x,\alpha)$.

We use a preordering $\leqslant$ on the set $P$ defined by the rule: $\alpha \leqslant \beta$ if and only if $B(x,\alpha)\subseteq B(x,\beta)$ for every $x\in X$. A subset $\P' \subseteq P$ is called {\it cofinal} if, for every $\alpha\in P$, there exists $\alpha' \in P'$ such that $\alpha \leqslant \alpha'$. A ballean $\B$ is called {\it ordinal} if there exists a cofinal subset $P'\subseteq P$ well ordered by $\leqslant$.

Let $\B = (X,P,B)$ be a ballean, $\mu$ be a cardinal. We say that a subset $A\subseteq X$ is {\it $\mu$-thin} if, for every $\alpha \in P$, there exists a bounded subset $Y\subseteq X$ such that $|B(x,\alpha)\cap A|\leqslant \mu$ for every $x\in G\setminus Y$. A 1-thin subset is called {\it thin}.

\begin{lemma}
Let $\B = (X,P,B)$ be a ballean, $\mu$ be a cardinal. A subset $A\subseteq X$ is $\mu$-thin if and only if the set 
$$\{a\in A: |B(a,\alpha) \cap A| > \mu\}$$
is bounded.
\end{lemma}

\begin{proof}
The "if" part is evident. To verify the "only if", we take an arbitrary $\alpha\in P$ and choose $\beta \in P$ such that $B(B(x,\alpha),\alpha)\subseteq B(x,\beta)$ for each $x\in X$. By the assumption, the set $Y = \{a\in A: |B(a,\alpha) \cap A| > \mu\}$ is bounded. We put $Z = B(Y,\alpha)$ and take an arbitrary $x\in X\setminus Z$. If $|B(x,\alpha)\cap A| > \mu$ and $a \in B(x,\alpha)\cap A$ then $|B(a,\beta)\cap A| > \mu$ because $B(x,\alpha)\subseteq B(a,\beta)$. Hence, $a\in Y$ and $x\in Z$. This contradiction shows that $|B(x,\alpha)\cap A| \leqslant \mu$ and $A$ is $\mu$-thin.
\end{proof}

\begin{theorem}
Let $\B = (X,P,B)$ be a ballean, $\mu$ be a cardinal, $A\subseteq X$. Then the following statements hold

\begin{itemize}
\item[(i)] if $A$ is a union of $\mu$ thin subsets and a union of $\mu$ bounded subsets of $X$ is bounded then $A$ is $\mu$-thin;

\item[(ii)] if $\B$ is ordinal and $A$ is $\mu$-thin, $\mu\in\mathbb{N}$ then $A$ can be partitioned into $\leqslant \mu$ thin subsets.
\end{itemize}
\end{theorem}

\begin{proof}
(i) Let $A=\bigcup_{\lambda \leqslant \mu} A_\lambda$ and each $A_\lambda$ is thin, $\alpha \in P$. For each $\lambda \leqslant \mu$, we pick a bounded subset $Y_\lambda$ such that $|B(x,\alpha)\cap A|\leqslant 1$ for each $x\in X\setminus Y_\lambda$. We put $Y=\bigcup_{\lambda \leqslant \mu} Y_\lambda$. By the assumption, $Y$ is bounded. Clearly, $|B(x,\alpha)\cap A| \leqslant \mu$ for each $x\in X\setminus Y$ so $A$ is $\mu$-thin.

(ii) Apply Lemma 1 and \cite[Theorem 1.2]{b3}.
\end{proof}

\begin{theorem}
Let $G$ be a group, $\kappa$ be an infinite regular cardinal, $|G| = \kappa$. Then the following statements hold

\begin{itemize}
\item[(i)] each $(\kappa, m)$-thin subset $A$ of $G$, $m\in \mathbb{N}$ is a union of $\leqslant m$ $(\kappa,1)$-thin subsets;

\item[(ii)] the ideal generated by the family of $(\kappa,1)$-thin subsets coincides with the family of all $(\kappa, m)$-thin subsets, $m\in \mathbb{N}$.
\end{itemize}
\end{theorem}

\begin{proof}
Clearly, (ii) follows from (i). To prove (i), we consider a ballean $\B(G,\kappa) = (G, [G]^{<\kappa}, B)$ where $B(x,F) = Fx\cup \{x\}$ for all $x\in G, F\in [G]^{<\kappa}$. We enumerate $G=\{g_\alpha: \alpha<\kappa\}$ and put $F_\alpha = \{g_\beta: \beta < \alpha\}$. Since $\kappa$ is regular, $\{F_\alpha: \alpha <\kappa \}$ is cofinal in $[G]^{<\kappa}$ so $\B(G,\kappa)$ is ordinal. To apply Theorem 1(i), it suffices to note that $A$ is $(\kappa, m)$-thin if and only if $A$ is $m$-thin in the ballean $\B(G,\kappa)$.
\end{proof}

In view of \cite[Chapter 1]{b8}, the balleans can be considered as asymptotic counterparts of the uniform spaces. For uniform spaces see \cite[Chapter 8]{b10}. Now we describe the uniform counterparts of thin subsets.

Let $\UU$ be a uniformity on a set $X$. For an entourage $U\in \UU$ and $x\in X$, we put $U(x) =\{y\in X: (x,y)\in U\}$. Let $A$ be a subset of $X$, $\mu$ be a cardinal. We say that $A$ is {\it $(\UU, \mu)$-discrete } if there exists $U\in \UU$ such that $|U(x) \cap A| \leqslant \mu$ for each $x\in X$, a $(\UU,1)$-discrete subset is called {\it $\UU$-discrete}. We show that each $(\UU, \mu)$-discrete subset $A$ of $X$ can be partitioned into $\leqslant \mu$ $\UU$-discrete subsets.

We fix an entourage $U\in \UU$ such that $|U(x)\cap A|\leqslant \mu$ for each $x\in X$ and choose a symmetric entourage $V\in \UU$ such that $V^2 \subseteq U$. Then we consider a graph $\Gamma$ with the set of vertices $A$ and the set of edges $E$ defined by the rule: $(x,y)\in E$ if and only if there exists $z\in X$ such that $x,y\in V(z)$. Since $|U(x)\cap A|\leqslant \mu$ for each $x\in X$ and $V^2 \subseteq U$, each unit ball in $\Gamma$ is of cardinality $\leqslant \mu$. Hence, the chromatic number $\chi(\Gamma)$ does not exceed $\mu$. We take a partition $\mathcal{P}$ of $A$ such that $|\mathcal{P}| = \chi(\Gamma)$ and each $P\in \mathcal{P}$ has no incident vertices. If $x\in P$ then $V(x)\cap P = \{x\}$. It follows that $P$ is $\UU$-discrete.

\section{Partitions}

\begin{lemma}
Let $G$ be a group, $\kappa$ be an infinite cardinal, $\kappa \leqslant |G|$, $m\in\mathbb{N}$. Let $A$ be a $(\kappa,m)$-thin subset of $G$, $S\subseteq G$, $|S|\geqslant \kappa$. Then there exists a subgroup $H$ of $G$ such that $S\subseteq H$, $|H| =|S|$ and $|Hx\cap A|\leqslant m$ for each $x\in G\setminus H$. In particular, $A$ is $(\kappa',m)$-thin for each $\kappa' \geqslant \kappa$, $\kappa' \leqslant |G|$.
\end{lemma}

\begin{proof}
We may suppose that $S$ is a subgroup. Let $H_0 = S$, $[H_0]^{m+1} = \{X\subset H_0: |X| = m+1\}$, $|S| = \kappa'$. For each $X\in [H_0]^{m+1} $, we choose $K_0(X) \in [G]^{<\kappa} $ such that $|Xg\cap A| \leqslant m$ for every $g\in G\setminus K_0(X)$. We put $K_0 = \cup\{K_0(X): X\in [H_0]^{m+1}\}$ and note that $|K_0|\leqslant \kappa'$ and $|H_0x\cap A|\leqslant m$ for each $x\in G\setminus K_0$.

We consider a subgroup $H_1$ of $G$ generated by $H_0\cup K_0$. Clearly, $|H_1| = \kappa'$. For each $X\in [H_1]^{m+1}$, we take $K_1(X) \in [G]^{<\kappa}$ such that $|Xg\cap A| \leqslant m$ for every $g\in G\setminus K_1(X)$. We put $K_1 = \cup\{K_1(X): X\in [H_1]^{m+1}\}$ and note that $|H_1x\cap A|\leqslant m$ for each $x\in G\setminus K_1$.

After $\omega$ steps we get an increasing sequence of $\{H_n: n\in \omega\}$ of subgroups of $G$ and a sequence $\{K_n:n\in \omega\}$ of subsets of $G$ such that $|H_n| = \kappa'$, $|K_n| \leqslant \kappa'$. Since $\cup_{n\in \omega} K_n \subseteq \cup_{n\in \omega} H_n$, for the subgroup $H = \cup_{n\in \omega} H_n$ we get a desired statement.

To show that $A$ is $(\kappa',m)$-thin, we take $S\in [G]^{<\kappa'}$. If $|S|<\kappa$ then there exists $K\in [G]^{<\kappa}$ such that $|Sx\cap A|\leqslant m$ for each $x\in G\setminus K$ because $A$ is $(\kappa,m)$-thin. If $|S|\geqslant \kappa$, we apply the previous statement.
\end{proof}

For a cardinal $\kappa$ and $n\in\omega$, we use the following notations from \cite{b11}: $\kappa^{+0} = \kappa, \kappa^{+(n+1)}= (\kappa^{+n})^+$. In particular, $\aleph_0^{+n} = \aleph_n$ for each $n\in \omega$.

\begin{theorem}
Let $\kappa$ be an infinite regular cardinal, $m\in\mathbb{N}$, $n\in\omega$, $G$ be a group of cardinality $\kappa^{+n}$. Each $(\kappa, m)$-thin subset $A$ of $G$ can be partitioned into $\leqslant m^{n+1}$ $(\kappa,1)$-thin subsets.
\end{theorem}

\begin{proof}
We use an induction by $n$. For $n=0$, this is Theorem 2. Let $A$ be a $(\kappa,m)$-thin subset of $G$ and $|G|=\kappa^{+(n+1)}$. By Lemma 2, $A$ is $(\kappa^{+(n+1)}, m)$-thin. Applying Theorem 2, we can partition $A$ in $\leqslant m$ $(\kappa^{+(n+1)}, 1)$-thin subsets. We suppose that $A$ itself is $(\kappa^{+(n+1)}, 1)$-thin and show that $A$ can be partitioned in $\leqslant m^{n+1}$ $(\kappa,1)$-thin subsets.

Since $A$ is $(\kappa^{+(n+1)}, 1)$-thin, we use Lemma 2 to write $G$ as a union of increasing chain of subgroups $\{H_\alpha: \alpha < \kappa^{+(n+1)}\}$ such that $H_0 =\{e\}$, $e$ is the identity of $G$, $|H_\alpha| = \kappa^{+n}$ for each $\alpha >0$, $H_\beta = \cup_{\alpha < \beta} H_\alpha$ for each limit ordinal $\beta < \kappa^{+(n+1)}$. Clearly, $G\setminus \{e\} = \cup_{\alpha < \kappa^{+(n+1)}} H_{\alpha+1}\setminus H_\alpha$.

For each $\alpha < \kappa^{+(n+1)}$, $\alpha >0$, we put $A_\alpha = A\cap H_\alpha$. Since $A_\alpha$ is $(\kappa,m)$-thin and $|H_\alpha| = \kappa^{+n}$, by the inductive assumption, each $A_\alpha$ can be partitioned in $k_\alpha \leqslant m^{n+1}$ $(\kappa,1)$-thin subsets of $H_\alpha$. Admitting empty sets of the partition, we suppose that $k_\alpha = m^{n+1}$ for each $\alpha < \kappa^{+(n+1)}$ and write
$$A_\alpha = A_\alpha(1) \cup \ldots \cup A_\alpha(m^{n+1}),$$
where each $A_\alpha(i)$ is $(\kappa,1)$-thin.

For all $\alpha < \kappa^{+(n+1)}$ and $i\in \{1,\ldots , m^{n+1} \}$, we put 

$$B_\alpha(i) = A_{\alpha+1}(i) \setminus A_\alpha(i), \quad B_i = \bigcup_{\alpha < \kappa^{+(n+1)}}B_\alpha(i).$$

Since $A\setminus \{e\} = \cup \{B_i: i\in \{1,\ldots , m^{n+1} \} \}$, it suffices to verify that each subset $B_i$ is $(\kappa,1)$-thin. It turns out, since $\kappa$ is regular, in view of Lemma 1, it suffices to show that, for each $g\in G$, $g\ne e$,

$$|\{x\in B_i: gx \in B_i\}|< \kappa.$$

We take the minimal $\alpha < \kappa^{+(n+1)}$ such that $g\in H_{\alpha+1} \setminus H_\alpha$. If $x\in B_i\setminus H_{\alpha +1}$ then $gx\notin B_i$ by the choice of $H_{\alpha +1}$. Since $A_{\alpha+1}(i)$ is $(\kappa,1)$-thin, $|\{x\in A_{\alpha+1}(i) \setminus A_\alpha(i) : gx\in A_{\alpha+1}(i) \setminus A_\alpha(i)\}|< \kappa$. If $x,y\in A_{\alpha+1}(i) \setminus A_\alpha(i)$ and $gx,gy\in A_\alpha(i)$ then $x^{-1}y \in A_\alpha(i)$ so, by the choice of $H_{\alpha }$, $x=y$. If $x,y\in A_\alpha(i)$ and $gx, gy \in A_{\alpha+1}(i) \setminus A_\alpha(i)$, replacing $g$ to $g^{-1}$, we get the previous case.

Theorem is proved.

\end{proof}

Given an infinite group $G$ and infinite cardinal $\kappa$, $\kappa \leqslant |G|$, we denote by $\mu(G,\kappa)$ the minimal cardinal $\mu$ such that $G$ can be partitioned in $\mu$ $\kappa$-thin subsets. For a cardinal $\gamma$, $cf\;\gamma$ is a cofinality of $\gamma$, $\gamma^+$ is the cardinal successor of $\gamma$. By \cite{b5},

$$
\mu(G,\kappa) = \left\{
\begin{array}{ll}
\gamma,& \mbox{if $|G|$ is non-limit cardinal and $|G| = \gamma^+$;}\\
|G|,& \mbox{if $|G|$ is a limit cardinal and either}\\
&\mbox{$\kappa < |G|$ or $|G|$ is regular} \\
cf\; |G|, &\mbox{if $|G|$ is singular, $\kappa = |G|$ and $cf\; |G|$}\\
&\mbox{is a limit cardinal} \\
\end{array} \right.
$$

If $|G|$ is singular, $\kappa = |G|$ and $cf\; |G|$ is a non-limit cardinal, $cf|G|= \gamma^+$ then $\mu(G,\kappa)\in \{\gamma, \gamma^+\}$.

Now let $\gamma$ be a cardinal, $\gamma \leqslant \kappa$. Then $G$ is $(\kappa, \gamma)$-thin if and only if $\kappa = \gamma^+$. Applying above formulae for $\mu(G,\kappa)$, we get the following statement. 

\begin{theorem}
Let $G$ be a group, $\gamma$ be an infinite cardinal, $\kappa = \gamma^+$, $|G|\geqslant \gamma$. Then $G$ can be partitioned in $\gamma$ $(\kappa,1)$-thin subsets if and only if $|G|=\kappa$.
\end{theorem}

\section{Colorings}

For a group $G$ and $g\in G$, we say that $\{G\}\times \{g\}$ is a {\it horizontal line} in $G\times G$, $\{g\}\times G$ is a {\it vertical line} in $G\times G$, $\{(x,gx): x\in G\}$ is a {\it diagonal} in $G\times G$.

The statement (i) in the following theorem was proved by G.~Bergman, (ii) by the first author.

\begin{theorem}
For a group $G$ with the identity $e$, the following statements hold

\begin{itemize}
\item[(i)] if $|G| \geqslant \aleph_2$ and $\chi: G\times G \to \{1,2,3\} $ then there is $g\in G$, $g\ne e$ such that either some horizontal line $G\times \{g\}$ has infinitely many points of color 1, or some vertical line $\{g\}\times G$ has infinitely many points of color 2, or some diagonal $\{(x,gx): x\in G\}$ has infinitely many points of color 3;

\item[(ii)]  if $|G| \leqslant \aleph_1$ then there is a coloring $\chi: G\times G \to \{1,2,3\} $ such that each horizontal line has only finite number of points of color 1, each vertical line has only finite number of points of color 2, each diagonal has only finite number of points of color 3.
\end{itemize}
\end{theorem}

\begin{proof}
(i) We suppose the contrary and fix a corresponding coloring $\chi: G\times G \to \{1,2,3\}$. Let $G_0, G_1$  be subgroups of $G$ such that $G_0\subset G_1$, $|G_0|=\aleph_0$, $|G_1|=\aleph_1$. Since the set $G\times G_1$ has at most $\aleph_1$ points of color 1, there is $g\in G$, $g\ne e$ such that $gG_1 \times G_1$ has no points of color 1. The set $gG_0 \times G_1$ has at most $\aleph_0$ points of color 2, so there is $h\in G_1$ , $h\ne g$, $h\ne e$ such that $gG_0\times hG_0$ has no points of color 2. Hence, the set $\{(gx, hx): x\in G_0\}$ consists of color 3 and is contained in the diagonal $\{(y, g^{-1}hy): y\in G\}$.

(ii) We proceed in three steps.

Step 1. Let $X$ be a countable set, $\{A_n:n\in \omega\}$, $\{B_n:n\in \omega\}$ be partitions of $X$ such that $A_n\cap B_m$ is finite for all $n,m\in\omega$. Then there is a coloring $\chi:X->\to\{1,2\}$ such that each subset $A_n$ has only finite number of elements of color 1 and each subset $B_n$ has only finite number of elements of color 2.

We color $A_0$ in 2, $B_0\setminus A_0$ in 1, $A_1\setminus B_0$ in 2, $B_1\setminus (A_0\cup A_1)$ in 1, $A_2\setminus (B_0\cup B_1)$ in 2, $B_2\setminus(A_0\cup A_1 \cup A_2)$ in 1, and so on.

Step 2. Let $H$ be a countable group, $K$ be a subgroup of $H$. Applying Step 1, we define a coloring $\chi:((H\times H)\setminus (K\times K))\to \{1,2,3\}$ such that 

\begin{itemize}
\item $\chi((H\setminus K)\times(H\setminus K))=\{1,2\}$ and each horizontal line in this set has only finite number of points of color 1, and each vertical line has only finite number of points of color 2.
\item $\chi(K\times(H\setminus K))=\{1,3\}$ and each horizontal line in this set has only finite number of points of color 1, and each diagonal has only finite number of points of color 3.
\item $\chi((H\setminus K)\times K)=\{2,3\}$ and each vertical line in this set has only finite number of points of color 2, and each diagonal has only finite number of points of color 3.
\end{itemize}

Step 3. To prove (ii), we may suppose that $|G| = \aleph_1$ so write $G$ as a union of an increasing chain $\{G_\alpha:\alpha < \omega_1\}$, $G_0=\{e\}$, $e$ is the identity of $G$, such that $G_\beta=\bigcup_{\alpha<\beta} G_\alpha$ for each limit ordinal $\beta<\omega_1$. We put $\chi_0(e)=1$ and, for each $\alpha < \omega_1$ use a coloring $\chi_\alpha((G_{\alpha+1}\times G_{\alpha+1})\setminus(G_\alpha\times G_\alpha))$ defined on Step 2. We put $\chi=\bigcup_{\alpha<\omega_1}\chi_\alpha$ and verify that $\chi: G\times G\to \{1,2,3\}$ is thin.

Clearly, $(G\times\{e\})\cap \chi^{-1}(1)\subseteq G_1\times\{e\}$, $(\{e\}\times G)\cap \chi^{-1}(2)\subseteq \{e\}\times G_1$. If $g\in G_{\alpha+1}\setminus G_\alpha$ then 
$$
(G\times\{g\})\cap \chi^{-1}(1)\subseteq G_{\alpha+1}\times\{g\},\quad (\{g\}\times G)\cap \chi^{-1}(2)\subseteq \{g\}\times G_{\alpha+1}.
$$
Thus, each horizontal line in $G\times G$ has only finite number of points of color 1, and each vertical line in $G\times G$ has only finite number of points of color 2.

At last, if $(x,y)\in (G_{\alpha+1}\setminus G_\alpha)\times G_\alpha$ or $(x,y)\in G_\alpha\times (G_{\alpha+1}\setminus G_\alpha)$ then $x^{-1}y\in G_{\alpha+1}\setminus G_\alpha$. It follows that each diagonal has only finite number of points of color 3.
\end{proof}

In \cite{b12}, \cite{b13} R.~Davies proved the following theorem (see also \cite[Theorem 1.7]{b11}). 

For every $n\in \mathbb{N}$, the following statements are equivalent

\begin{enumerate}
\item $2^{\aleph_0}\leqslant \aleph_n$;
\item There is a sequence $L_0,\ldots L_{n+1}$ of lines in the plane $\mathbb{R}^2$ and a coloring $\chi: \mathbb{R}^2 \to \{0,\ldots, n+1\}$ such that, for each $i\in \{0,\ldots, n+1\}$, every line in $\mathbb{R}^2$ parallel to $L_i$ intersects $\chi^{-1}(i)$ in finitely many points.
\end{enumerate}

We note that the group $\mathbb{R}$ of real numbers is the direct sum of $2^{\aleph_0}$ copies of the group $\mathbb{Q}$ of rational numbers and use a part of this theorem in the following form.

\begin{theorem}
Let $n\in \mathbb{N}$, $H=\oplus_{\aleph_n}\mathbb{Q}$, $a_0, \ldots ,a_n$, $b_0, \ldots ,b_n$ be rational numbers such that, for each $i$, either $a_i \ne 0$ or $b_i \ne 0$. Then, for every coloring $\chi: H\times H\to \{0,\ldots, n\}$, there exist $i\in \{0,\ldots, n\}$, $h\in H$, $h\ne 0$ and infinitely many pairs $a,b\in H$ such that $\chi(a,b)=i$ and $a_ia+b_ib = h$.
\end{theorem}

\begin{proof}
For $i\in \{0,\ldots, n\}$, let $L_i = \{(x,y)\in H\times H : a_ix + b_iy = 0\}$. Apply Davies' theorem to the lines $L_0,\ldots L_{n}$.
\end{proof}

\section{Examples}

With minor changes, the first example is a reprint of the original Bergman's letter to the first author (15 May 2011).

\begin{example}
Let $H,K$ be groups of cardinality $\aleph_2$, $G = H\times K$. We construct a 2-thin subset $A$ of $G$ which cannot be partitioned into two thin subsets.

One can find a thin subset $X\subset K$ of cardinality $\aleph_2$. For instance, do a recursion over $\aleph_2$, selecting at each step an element not in the subgroup generated by those that precede. Since $X$ has cardinality $\aleph_2$, we can index it by pairs of elements of $H$: $X = \{x_{\{a,b\}}: a,b\in H\}$. After choosing such an indexing, we let 
$$A = \{x_{\{a,b\}}, ax_{\{a,b\}}, bx_{\{a,b\}}: a,b \in H\}, \quad A\subseteq H\times K = G.$$

I claim first that $A$ is 2-thin. For this it suffices to show that for every 3-element subset $F$ of $G$, only finitely many right translates of $F$ lie in $A$. In proving this, we may, by an initial right translation assume that $e\in F$, $e$ is the identity of $G$.

Assume that $F$ lay in $HK$ but not in $H$. Then every one of its right translates $Fg: g\in G$ has elements lying in more than one left coset of $H$; hence if such a right translate is contained in $A$, its elements do not all have the same second coordinate in $X$. Since $X$ is thin in $K$, we have $\{g\in G: Fg\subset A\}$ is finite.

We are left with the case $F\subset H$. In this case, it is not hard to see that $F$ has exactly 6 right translates contained in $A$; namely, these are obtained by taking the 6 arrangement of the elements of $F$ as an ordered 3-tuple, applying to each the right translate by a member of $H$ that puts it in the form $(e,a,b)$ and then right translating this by $x_{\{a,b\}}$ to get a 3-tuple of members of $A$.

Finally, let us show that $A$ cannot be partitioned into two thin subsets, $A_1$ and $A_2$. Suppose we had such a partition. Then let us color $H\times H$ as follows. Color an element $(a,b)\in H\times H$

\begin{itemize}
\item  with color 1 if $x_{\{a,b\}}$ and $ax_{\{a,b\}}$ lie in the same one of $A_1$ and $A_2$,  and  $bx_{\{a,b\}}$  lies in the other one.
\item  with color 2 if $x_{\{a,b\}}$ and $bx_{\{a,b\}}$ lie in the same one of $A_1$ and $A_2$,  and  $ax_{\{a,b\}}$  lies in the other one. 
\item  with color 1 if $ax_{\{a,b\}}$ and $bx_{\{a,b\}}$ lie in the same one of $A_1$ and $A_2$,  and  $x_{\{a,b\}}$  lies in the other one. 
\item  with any of these three colors if  $x_{\{a,b\}}$, $ax_{\{a,b\}}$, and $bx_{\{a,b\}}$ all lie in the same set $A_1$ or $A_2$.
\end{itemize}

Now if some horizontal line  $\{a\} x H$ in $H\times H$ had infinitely many points of color 1, then there would be infinitely many $b$ such that $x_{\{a,b\}}$ and $ax_{\{a,b\}}$ lay in the same one of $A_1$ and $A_2$. Hence one of the latter sets, say $A_i$, contains infinitely many 2-element sets $\{ x_{\{a,b\}}, ax_{\{a,b\}}\}$.  This gives infinitely many right translates of the pair $\{1, a\}$ in $A_i$, contradicting the assumption of thinness.

If some horizontal  line $H\times \{a\}$ or diagonal $\{(h,ah): h\in H\}$ had infinitely many points of color 2, respectively 3, we would get a contradiction in the same way. The case of horizontal line is like that of vertical line, so let us check the diagonal case. Suppose that for some $a$ infinitely many of the pairs $hx_{\{h,ah\}}$ and $ahx_{\{h,ah\}}$ lay in the same of our thin sets. Then at least one of those sets would contain infinitely many of these pairs; but this shows that the set would contain infinitely many right translates of the pair $\{e,a\}$, contradicting thinness.

The above arguments show that our coloring of $H\times H$ contradicts Theorem 5(i); so $A$ cannot, as assumed, be decomposed into two thin sets.
\end{example}

\begin{example}
For each natural number $m\geqslant 2$, we construct an Abelian group $G$ of cardinality $\aleph_n$, $ n = \frac{m(m+1)}{2} -1$, and find a 2-thin subset $A$ of $G$ which cannot be partitioned into $m$-thin subsets.

We put $H=\oplus_{\aleph_n}\mathbb{Q}$, take an arbitrary Abelian group $K$ of cardinality $\aleph_n$ and let $G=K\oplus H$. Then we choose a thin subset $X$ of $K$, $|X| = \aleph_2$ and enumerate $X$ by the pairs of elements of $H$: $X = \{x_{\{a,b\}}: a,b\in H\}$. After choosing such an indexing, we let 

$$A = \{x_{\{a,b\}} + ka + k^2b: a,b \in H , k \in \{0,\ldots,m\}  \}.$$

To see that $A$ is 2-thin, it suffices to show that, for any two distinct non-zero elements $x,y\in G$, the set

$$A(x,y)=\{a\in A: a+x\in A, a+y\in A\}$$
is finite. We write $x = x_1 + x_2$, $y = y_1 + y_2$, $x_1,x_2\in K$, $y_1,y_2\in H$. If either $x_1 = 0$ or $x_2 = 0$, $A(x,y)$ is finite because $X$ is thin. Let $x_1 = x_2 = 0$. If $x_{\{a,b\}} + ia + i^2b\in A(x,y)$ then

$$
\left\{
\begin{array}{lll}
x_2 + ia + i^2b &  = & ja+j^2b \\
y_2 + ia + i^2b &  = & ka+k^2b\\
\end{array} \right.
$$
for some distinct $i,k\in \{0,\ldots,m\} $. In this system of relations, $a,b$ are uniquely determined by $i,j,k$. Since we have only finite number of possibilities to choose $i,j,k$, $A(x,y)$ is finite.

Now assume that $A$ is partitioned $A=A_1\cup \ldots \cup A_m$. To show that at least one cell of the partition is not thin, we define a coloring $\chi: H\times H\to \{(k,l): 0\leqslant k < l \leqslant m\}$ by the following rule: for $a,b\in H$, we choose $k,l$ so that $x_{\{a,b\}} + ka + k^2b$, $x_{\{a,b\}} + la + l^2b$ lie in the same cell of the partition $A_1\cup \ldots \cup A_m$. We note that 

$$(x_{\{a,b\}} + ka + k^2b) - (x_{\{a,b\}} + la + l^2b) = (k-l)a + (k^2-l^2)b.$$

Since $\frac{m(m+1)}{2} = n+1$, by Theorem 6, there exist $h\in H, h\ne 0, k<l$, and infinitely many monochrome pairs $a,b$ such that 

$$(k-l)a + (k^2-l^2)b = h.$$

By the definition of $\chi$, there are a cell $A_i$ of the partition and infinitely many pairs $a,b$ such that 

$$x_{\{a,b\}} + ka + k^2b \in A_i, x_{\{a,b\}} + ka + k^2b + h\ in A_i,$$
so $A_i$ is not thin.
\end{example}

\begin{example}
We construct a group $G$ of cardinality $\aleph_\omega$ and point out a 2-thin subset $A$ of $G$ which cannot be partitioned into $m$ thin subsets for each $m\in\mathbb{N}$.

For each $m\geqslant 2$, we take a group $G_n$, $ n = \frac{m(m+1)}{2} -1$ from Example 2, put $N = \{ \frac{m(m+1)}{2} -1: m\geqslant 2 \}$, take a 2-thin subset $A_n$ of $G_n$ which cannot be partitioned into $m$-thin subsets, and denote

$$G = \oplus_{n\in N} G_n,\quad A = \cup_{n\in N}A_n.$$
We take any distinct $x,y\in G\setminus\{0\}$ and, in notation of Example 2, show that $A(x,y)$ is finite, so $A$ is 2-thin. If $x,y\in G_n$ for some $n$ then $A(x,y)$ is 2-thin because $A_2$ is 2-thin. If $x\notin G_n$ for each $n$ then $|A\cap (A+x)|\leqslant 1$ so $A(x,y)$ is also finite. By the choice of $\{A_n: n\in N\}$, $A$ cannot be finitely partitioned into thin subsets.
\end{example}

\section{Ultrafilter context}

Let $G$ be a discrete group, $\beta G$ be the Stone-\v{C}ech compactification of $G$.
We take the elements of $\beta G$ to be ultrafilters on $G$ identifying $G$ with the set of principal ultrafilters, so $G^* = \beta G\setminus G$ is the set of free ultrafilters. The topology of $\beta G$ can be defined by the family $\{\overline{A}: A\subseteq G \}$ as a base for open sets, $\overline{A} = \{p\in \beta G: A\in p\}$.

The multiplication on $G$ can be naturally extended to $\beta G$ (see \cite[Chapter 4]{b14}). By this extension, the product $pq$ of ultrafilters $p$ and $q$ can be defined as follows. Take an arbitrary $P\in p$ and, for each $g\in P$, pick $Q_g\in q$. Then $\cup_{g\in P} gQ_g \in pq$ and each member of $pq$ contains a subset of this form. In particular, if $g\in G$ and $q\in \beta G$ then $gq = \{gQ: Q\in q\}$.

The proofs of all propositions in this section can be easily extracted from corresponding definitions.

\begin{proposition} For each $m\in \mathbb{N}$, a subset $A$ of a group $G$ is $m$-thin if and only if $|Gp\cap \overline{A}|\leqslant m$ for each $p\in G^*$.
\end{proposition}

By \cite{b1}, a subset $A$ of a group $G$ is {\it sparse} if for any infinite subset $X\subset G$ there exists a finite subset $F$ such that $\cap_{g\in F}gA$ is finite. An ultrafilter $p\in G^*$ has a sparse member if and only if $p\notin \overline{G^*G^*}$.

\begin{proposition} A subset $A$ of a group $G$ is sperse if and only if the set $Gp\cap \overline{A}$ is finite for each $p\in G^*$.
\end{proposition}

\begin{proposition} A subset $A$ of a group $G$ can be partitioned in finite number of thin subsets if and only if each ultrafilter $p\in \overline{A}$ has a thin member.
\end{proposition}

Now take a group $G$ of cardinality $\aleph_\omega$ from Example 3 and corresponding 2-thin subset $A$. Since $A$ cannot be finitely partitioned into thin subsets, by Proposition 3, there is $p\in \overline{A}$ with no thin members. Hence, $p$ has a base consisting of 2-thin but not thin subsets.

\end{document}